\documentclass[11pt,twoside]{article}
\usepackage{amsfonts}
\usepackage{amsmath}
\usepackage{amssymb}
\usepackage{multicol}
\usepackage{graphicx}
\usepackage{stmaryrd}
\usepackage{cite}
\usepackage{epsfig}

\newtheorem{theorem}{Theorem}[section]
\newtheorem{lemma}[theorem]{Lemma}

\newtheorem{example}[theorem]{Example}
\newtheorem{definition}[theorem]{Definition}
\newtheorem{remark}[theorem]{Remark}

\numberwithin{equation}{section}
\newenvironment{proof}[1][Proof]{\noindent\textbf{#1.} }{\hfill $\Box$}
\allowdisplaybreaks

 \makeatletter
\begin{document}
\title{\textbf{Traveling Wave Solutions for Delayed Reaction-Diffusion Systems and Applications to Diffusive Lotka-Volterra Competition Models with Distributed Delays}}
\author{{{\sc Guo Lin}$^{1,}$\thanks{E-mail: ling@lzu.edu.cn}  \ and \
{\sc Shigui Ruan}$^{2,}$\thanks{E-mail: ruan@math.miami.edu} }\\[2mm]
$^1${\it School of Mathematics and Statistics, Lanzhou University,}\\
{\it Lanzhou, Gansu 730000, People's Republic of China}\\
$^2${\it Department of Mathematics, University of Miami,}\\
{\it P. O. Box 249085, Coral Gables, FL 33124-4250, USA}}
\maketitle

\begin{abstract}
This paper is concerned with the traveling wave solutions of delayed reaction-diffusion systems. By using Schauder's fixed point theorem, the existence of traveling wave solutions is reduced to the existence of generalized upper and lower solutions. Using the technique of contracting rectangles, the asymptotic behavior of traveling wave solutions for delayed diffusive systems is obtained. To illustrate our main results,  the existence,
nonexistence and asymptotic behavior of positive traveling wave solutions of diffusive Lotka-Volterra competition systems with distributed
delays are established. The existence of nonmonotone traveling wave solutions of diffusive Lotka-Volterra competition systems is also discussed.
In particular, it is proved that  if there exists instantaneous self-limitation effect, then the large delays appearing
in the intra-specific competitive terms may not affect the existence and asymptotic behavior  of traveling wave solutions.


\end{abstract}

\section{Introduction}
\noindent

In studying the nonlinear dynamics of delayed reaction-diffusion systems, one of the important topics is the existence of traveling wave solutions because of their significant roles in biological invasion and epidemic spreading,  we refer to Ai \cite{ai}, Gourley and Ruan \cite{gourleyruan}, Fang and Wu \cite{fang}, Faria et al. \cite{fa}, Faria and Trofimchuk \cite{fa1,fa2}, Huang and Zou \cite{huangzou2}, Kwong and Ou \cite{ko}, Li et al. \cite{llr}, Liang and Zhao \cite{liangzhao}, Lin et al. \cite{llm}, Ma \cite{ma01,ma1}, Ma and Wu \cite{mawu}, Mei \cite{mei1}, Ou and Wu \cite{ou}, Schaaf \cite{schaaf}, Smith and Zhao \cite{smithzhao}, Thieme and Zhao \cite{t}, Wang \cite{wanghy}, Wang et al. \cite{wang1,wangli1,wang3}, Wu and Zou \cite{wuzou2}, and Yi et al. \cite{yi}. It has been shown that delay may induce some differences of  traveling wave solutions between the delayed and undelayed systems, for example, the minimal wave speed (see Schaaf \cite{schaaf}, Zou \cite{jcam}) and the monotonicity of traveling wave solutions in scalar equations (see an example in Faria and Trofimchuk \cite{fa1}). In particular, the asymptotic behavior of traveling wave solutions, which is often formulated by the asymptotic boundary conditions, plays a very crucial role since it describes the propagation processes in different natural environments. For example, with proper asymptotic boundary value conditions, the traveling wave solutions of two species diffusive competition systems may reflect the coinvasion-coexistence of two invaders (see Ahmad and Lazer \cite{ah}, Li et al. \cite{llr}, Lin et al. \cite{llm}, Tang and Fife \cite{ta}), or exclusion between an invader and a resident (see Gourley and Ruan \cite{gourleyruan}, Huang \cite{huang}). Therefore, understanding the asymptotic behavior is a fundamental issue in the study of traveling wave solutions. Moreover, very detailed asymptotic behavior of traveling wave solutions has also been studied due to the development in mathematical theory of delayed reaction-diffusion systems, such as in the asymptotic stability and uniqueness of traveling wave solutions (see Mei et al. \cite{mei1}, Volpert et al. \cite{volpert}, Wang et al. \cite{wang3}).

To obtain the asymptotic behavior of traveling wave solutions in delayed reaction-diffusion systems, there are several methods. The first is based on the monotonicity of traveling wave solutions, which is often considered under the assumption of quasimonotonicity  in the sense of proper ordering (Huang and Zou \cite{huangzou2}, Ma \cite{ma01}, Wang et al. \cite{wang1}, Wu and Zou \cite{wuzou2}). The second is to construct proper auxiliary functions, such as the upper and lower solutions (Li et al. \cite{llr}, Lin et al. \cite{llm}). When these methods fail, a fluctuation technique is utilized to study the asymptotic behavior of traveling wave solutions if the system satisfies the locally quasimonotone condition near the unstable steady state (Ma \cite{ma1}, Wang \cite{wanghy}). Some other results have also been presented for (fast) traveling wave solutions of  scalar equations (Faria and Trofimchuk \cite{fa1,fa2}, Kwong and Ou \cite{ko}).

Of course, before considering the asymptotic behavior of traveling wave solutions of delayed systems, we must establish the existence of nontrivial traveling wave solutions. To obtain the existence of traveling wave solutions of delayed systems, Wu and Zou \cite{wuzou2} used a monotone iteration scheme if a delayed system is cooperative in the sense of proper ordering, see also Ma \cite{ma01}. Li et al. \cite{llr} further considered the existence of traveling wave solutions in competition systems by a cross iteration technique. Very recently, Ma \cite{ma1} studied the traveling wave solutions of a locally monotone delayed equation by constructing auxiliary monotone equations, see also Wang \cite{wanghy}, Yi et al. \cite{yi}. Moreover, by regarding the time delay as a parameter, the existence of traveling wave solutions was also studied by the perturbation method or Banach fixed point theorem, see Gourley and Ruan \cite{gourleyruan}, Ou and Wu \cite{ou}.

In this paper, we first study the existence and asymptotic behavior of traveling wave solutions in the following delayed reaction-diffusion system
\begin{equation}\label{cfu}
\frac{\partial v_i(x,t)}{\partial t}=d_i\Delta
v_i(x,t)+f_i(v_t(x)),
\end{equation}
where $x\in\mathbb{R},t>0,$ $v=(v_1, v_2, \cdots, v_n)\in \mathbb{R}^n, d_1,d_2, \cdots, d_n$ are positive  constants, $v_t(x):=v(x,t+s), s\in [-\tau, 0]$ and $\tau >0$ is the time delay, therefore, $f_i: C([-\tau, 0], \mathbb{R}^n)\to \mathbb{R},$ here $C([-\tau, 0], \mathbb{R}^n)$ is the space of continuous functions defined on $[-\tau, 0]$ and valued in $\mathbb{R}^n,$ which is a Banach space equipped with the supremum norm.

To overcome the difficulty arising from the deficiency of comparison principle, we introduce the definition of generalized upper and lower solutions of the corresponding wave system of delayed system \eqref{cfu}.  By using Schauder's fixed point theorem, the existence of traveling wave solutions is reduced to the existence of generalized upper and lower solutions. Motivated by the idea of contracting rectangles in evolutionary systems, we establish an abstract conclusion on the asymptotic behavior of positive traveling wave solutions in general partial functional differential equations. Subsequently, we consider two scalar delayed equations with diffusion in population dynamics (see Ma \cite{ma1} and Zou \cite{jcam}), which implies that our methods can also be applied to some well studied models.

In population dynamics, the following Lotka-Volterra reaction-diffusion system with distributed delay has been widely studied
\begin{equation}
\dfrac{\partial u_{i}(x,t)}{\partial t}=d_{i}\Delta
u_{i}(x,t)+r_{i}u_{i}(x,t)\left[ 1-\sum_{j=1}^{n}c_{ij}\int_{-\tau
}^{0}u_{j}(x,t+s)d\overline{\eta} _{ij}(s)\right] ,  \label{1}
\end{equation}
in which $i\in \{1,2,\cdots,n\}=:I,$ $x\in \Omega\subseteq
\mathbb{R}^k,k\in\mathbb{N}, t>0,$ $u=(u_1,u_2,
\cdots, u_n)\in\mathbb{R}^n,$ $u_i(x,t)$ denotes the density
of the $i-$th competitor at time $t$
and in location $x\in \Omega,$ $d_i>0,r_i>0, c_{ii}>0$ and
$c_{ij}\ge 0$ are constants for $i,j\in I,i\neq j.$ We also suppose that
\[
\overline{\eta} _{ij}(s) \text{  is nondecreasing on }[-\tau,0]
\text{ and }\overline{\eta} _{ij}(0)-\overline{\eta} _{ij}(-\tau)=1,
\]
which will be imposed throughout the paper. Let
\[
a_i=\overline{\eta}_{ii}(0)-\overline{\eta}_{ii}(0-), i\in I,
\]
and set
\[
{\eta }_{ii}(s)= \begin{cases}\overline{\eta} _{ii}(s),s\in \lbrack
-\tau ,0),\\
\overline{\eta} _{ii}(0-),s=0, \end{cases}{\eta
}_{ij}(s)=\overline{\eta }_{ij}(s), i,j\in I,i\neq j.
\]
Clearly, $a_i>0$ implies the existence of instantaneous self-limitation effect in population dynamics.

The dynamics of \eqref{1} has been studied by several authors, for
example, Gourley and Ruan \cite{gourleyruan}, Fang and Wu \cite{fang}, Li et al. \cite{llr},
Lin et al. \cite{llm}, Martin and Smith \cite{martin2} and Ruan and Wu \cite{rw}. More
precisely, when $\Omega$ is a bounded domain and \eqref{1} is equipped with
the Neumann boundary condition, Martin and Smith \cite{martin2} proved
that if the initial values of \eqref{1} are positive and
\begin{equation}\label{2}
\sum_{j=1}^n c_{ij} (c_{jj}a_j)^{-1}<2, i\in I,
\end{equation}
then the unique mild  solution to \eqref{1} satisfies
\begin{equation}\label{c}
u_i(x,t)\to u_i^*, t\to \infty, i\in I, x\in \Omega,
\end{equation}
hereafter $u^*=(u_1^*,u_2^*,\cdots,u_n^*)$ is the unique spatially
homogeneous positive steady state of \eqref{1}, of which the existence can be
obtained by \eqref{2}. It is well known that  $a_i<1$ implies the existence of
time delay in intra-specific competition, which often leads to some
significant differences between the dynamics of delayed and
undelayed models if the time delay is large. For  example, the
following Logistic and Hutchinson equations
\[
\frac{du(t)}{dt}=u(t)(1-u(t)),\,\,\,\,\,\,\frac{du(t)}{dt}=u(t)(1-u(t-\tau
)),\,\,\,\,\,\,\tau >0
\]
exhibit dramatically different dynamics, we refer to Ruan \cite{ruan} for detailed analysis
on these two equations and Hale and
Verduyn Lunel \cite{hale} and Wu \cite{wu} for
fundamental theories on delayed equations.
However, \eqref{c} does not depend on the size of delays and
the distribution of $\eta_{ii}(s)$ for $s\in [-\tau,0)$. Similar
phenomena can be founded in the corresponding functional
differential equations, see Smith \cite[Section 5.7]{smith}.

In particular, Li et al. \cite{llr} and Lin et al. \cite{llm} have established the existence
of traveling wave solutions to \eqref{1}, which models the
invasion-coexistence scenario of multiple competitors. However,
these results only hold if there exists $\tau_0\in [0,\tau]$ small
enough satisfying $\int_{-\tau_0 }^{0}d\overline{\eta} _{ii}(s)=1,
i\in I,$ which ensures the so-called exponentially monotone
condition such that the upper and lower solutions are admissible. If $\tau_0$ is large, then we cannot apply these techniques and results to study the existence and asymptotic behavior of traveling wave solutions of  \eqref{1}.  In this paper, we shall consider the existence and further properties of nontrivial
traveling wave solutions of \eqref{1} with $a_i>0,i\in I$.

Using the results on the existence and asymptotic behavior of traveling wave solutions in system \eqref{cfu}, model \eqref{1} with $\Omega =\mathbb{R}$ is studied by presenting the existence, nonexistence and asymptotic behavior of positive traveling wave solutions. In particular, due to less requirements for auxiliary functions, we  obtain some sufficient conditions on the existence of nonmonotone traveling wave solutions of  \eqref{1} with $a_i=1, i\in I$. Note that these
conclusions remain true if $\tau=0,$ we thus confirm  the conjecture about the existence of nonmonotone traveling wave solutions of competitive systems, which was proposed by Tang and Fife \cite[the last paragraph]{ta}.

The rest of this paper is organized as follows. In section 2, we list some preliminaries. Using  contracting rectangles, the asymptotic behavior of traveling wave solutions of  general partial functional differential equations \eqref{cfu} is established in section 3, and is applied to two examples  considered by Ma \cite{ma1} and Zou \cite{jcam}. In section 4, we introduce the generalized upper and lower solutions and study the existence of traveling wave solutions in \eqref{cfu}. In section 5, we investigate the  traveling wave solutions of the Lotka-Volterra system \eqref{1}, including the existence, nonexistence, asymptotic behavior and monotonicity. This paper ends with a brief discussion of our methods and results.

\section{Preliminaries}
\noindent

In this paper, we shall use the standard partial ordering and interval
notations in $\mathbb{R}^n.$ Namely, if $u=(u_1,u_2,\cdots,
u_n),v=(v_1,v_2,\cdots, v_n)\in\mathbb{R}^n, $ then $u\ge v$ iff $u_i\ge v_i, i\in
I;$ $u> v$ iff $u\ge v $ but $ u_i > v_i$ for some $i\in I;$ $u\gg
v$ iff $u_i> v_i, i\in I.$ Moreover, $X$ will be interpreted as
follows
\[
X=\{ u: u\text{ is a bounded and uniformly continuous function from
}\Bbb{R}\text{ to }\Bbb{R}^n\},
\]
which is a Banach space equipped with the supremum norm $\| \cdot \|$. If $a,b\in
\mathbb{R}^n$ with $a\le b,$ then
\[
X_{[a,b]}=\{u \in X: a\le u(x)\leq b, x\in \mathbb{R}\}.
\]
Let $u(x)=(u_1(x), \cdots,u_n(x) ), v(x)=(v_1(x), \cdots,v_n(x) )\in X,$
then $u(x)\ge v(x)$ implies that $u(x)\ge v(x) $ for all $x\in
\mathbb{R};$ $u(x)> v(x)$ is interpreted as $u(x)\ge v(x)$ but
$u(x)>v(x)$ for some $x\in \mathbb{R};$ and $u(x)\gg v(x)$  if $u(x)>v(x)$  and for each $i\in I,$ there exists
$x_i\in  \mathbb{R}$ such that $u_i(x_i)>v_i(x_i).$ $u(x)$ is a nonnegative, positive
and strictly positive function iff $u(x)\ge 0, u(x)>0$ and $u(x)\gg 0$, respectively.

Consider the Fisher equation
\begin{equation}\label{0.1}
\begin{cases}
\frac{\partial z(x,t)}{\partial t}=d \Delta z(x,t)+rz (x,t)\left[
1-z (x,t)\right] ,\\
z(x,0)=z(x)\in X_{[0,1]}
\end{cases}
\end{equation}
with $x\in\mathbb{R}, t>0, d>0,r>0$ and $n=1$ in the definition of $X$.
\begin{lemma}\label{le1.1}
For \eqref{0.1}, we have the following conclusions.
\begin{enumerate}
\item[(i)] \eqref{0.1} admits a unique solution $z(\cdot,t)\in
X_{[0,1]}$ for all $t>0.$
\item[(ii)] If $\overline{z}(\cdot,t), \underline{z}(\cdot,t)\in X$ for $ t>0$ such that
that
\begin{eqnarray*}
\frac{\partial \overline{z}(x,t)}{\partial t} &\geq &d\Delta \overline{z}%
(x,t)+r\overline{z}(x,t)\left[ 1-\overline{z} (x,t)\right] , \\
\frac{\partial \underline{z}(x,t)}{\partial t} &\le &d\Delta \underline{z}%
(x,t)+r\underline{z}(x,t)\left[
1-\underline{z} (x,t)\right] ,
\end{eqnarray*}
then $\overline{z}(x,t)$ and $ \underline{z}(x,t)$ are upper and
lower solutions to \eqref{0.1}, respectively. Furthermore, we have
\[
\overline{z}(x,t)\ge z(x,t)\ge  \underline{z}(x,t)
\]
if $\overline{z}(x,0)\ge z(x)\ge  \underline{z}(x,0).$
\item[(iii)]  If $z(x)>0$ and  $c\in (0,2\sqrt{dr}),$  then
\[
\liminf_{t\to\infty}\inf_{|x|<ct}z(x,t)=\limsup_{t\to\infty}\sup_{|x|<ct}z(x,t)=1.
\]
\end{enumerate}
\end{lemma}

For (i) and (ii) of Lemma \ref{le1.1}, we refer to  Fife \cite{fife} and Ye et al. \cite{yeli}. (iii) of Lemma \ref{le1.1} is the classical theory of asymptotic spreading, see Aronson and Weinberger \cite{aron2}.

\section{Asymptotic Behavior of Traveling Wave Solutions}
\noindent

Consider the following functional differential equation corresponding to \eqref{cfu}
\begin{equation}\label{f}
\frac{dl_i(t)}{dt}=f_i(l_{t}), i\in I,
\end{equation}
in which $l=(l_1, l_2, \cdots, l_n)\in\mathbb{R}^n,$ $f_i$ is defined by \eqref{cfu} and satisfies the following assumptions:
\begin{description}
\item[(H1)] there exists ${E}\gg {0}$ such that $f_i(\widehat{{0}})=f_i(\widehat{{E}})=0, $ where $\widehat{\cdot}$ denotes the constant valued function in $C([-\tau, 0], \mathbb{R}^n)$ and ${E}=(E_1, E_2,  \cdots, E_n)\in\mathbb{R}^n;$
\item[(H2)] there exist $\overline{{E}}\ge  {E}\ge \underline{{E}} \ge   {0}$ such that
$[\widehat{\underline{{E}}},\widehat{\overline{{E}}} ]$ is a positively invariant ordered interval of \eqref{f}, where
\[
\overline{{E}}=(\overline{E}_1,\overline{E}_2, \cdots, \overline{E}_n),\,\, \underline{{E}}=(\underline{E}_1, \underline{E}_2, \cdots, \underline{E}_n);
\]
\item[(H3)] if $u\in C([-\tau, 0], \mathbb{R}^n)$ and $\widehat{{0}}\le u(t+s)\le \widehat{b(0)}$ for $s\in [- \tau, 0],$ then $f_i(u_t):C([-\tau, 0], \mathbb{R}^n) \to \mathbb{R} $ is Lipschitz continuous in the sense of supremum norm, here $b(0)\ge \overline{{E}}$ is a constant vector clarified by (H4)-(H5);
\item[(H4)]  there exists a one-parameter family of ordered intervals given by
\[
\sum (y)=[\widehat{{a}}(y),\widehat{{b}}(y)]
\]%
such that $a(0)\le \underline{E}\le \overline{{E}}\le b(0)$ and for $0\leq y_{1}\leq y_{2}\leq 1$
\[
{0}\le {a}(0)\leq {a}(y_{1})\leq {a}(y_{2})\leq
{a}(1)={E}={b}(1)\leq {b}(y_{2})\leq {b}%
(y_{1})\leq {b}(0),
\]%
where ${a}(y)$ and ${b}(y)$ are continuous in $y\in [ 0,1]$;
\item[(H5)] $\sum (y)$ is a \emph{strict} contracting rectangle, namely, let
\[
{a}(y)=(a_1(y), a_2(y), \cdots, a_n(y)), \,\, {b}(y)=(b_1(y), b_2(y), \cdots, b_n(y)),
\]
then for any $y\in (0,1)$ and  $u\in \sum (y),$ we have
\[
f_i(u) >0  (f_i(u) < 0)\text{ if }u_i(0)=a_i(y) (u_i(0)=b_i(y)), i\in I.
\]
\end{description}

To continue our discussion, we now introduce the following definition of traveling wave solutions of \eqref{cfu}.
\begin{definition}{\rm
A \emph{traveling wave solution} of \eqref{cfu} is a special solution
\[
v_i(x,t)=\phi_i(x+ct), i\in I,
\]
where $c>0$ is the wave speed and $\Phi= (\phi_1,\phi_2, \cdots,\phi_n)\in C^2(\mathbb{R},\mathbb{R}^n)$ is the wave profile.}
\end{definition}

By the definition, $\Phi= (\phi_1,\phi_2,\cdots, \phi_n)$ satisfies
\begin{equation}\label{fu}
d_i\phi''_i(\xi)-c\phi'_i(\xi)+f_i^c(\Phi_{\xi})=0,i\in I,\xi\in\mathbb{R};
\end{equation}
in which $f_i^c(\Phi_{\xi}):C([-c\tau, 0],\mathbb{R}^n)\to \mathbb{R}$ is defined by
\[
f_i^c(\Phi_{\xi})=f_i(\Phi (\xi+ cs)),  s\in [-\tau, 0], i\in I.
\]

Using the contracting rectangles, we present the following asymptotic boundary conditions of traveling wave solutions.
\begin{theorem}\label{t1-cr}
Assume (H1)-(H5). Let $\Phi \in X$ be a positive solution of \eqref{fu} with
\begin{equation}\label{aa}
\overline{{E}}\gg \limsup_{\xi \to \infty } \Phi (\xi )\ge \liminf_{\xi \to \infty } \Phi (\xi ) \gg {\underline{E}}.
\end{equation}
Then $\lim_{\xi \to \infty}\Phi (\xi)={E}$ if $\Phi'$ and $\Phi''$ are uniformly bounded.
\end{theorem}

\begin{proof}
Denote
\[
\limsup_{\xi \to \infty } \Phi (\xi )=\Phi^+, \liminf_{\xi \to \infty } \Phi (\xi )=\Phi^-
\]
with
\[
\Phi^{\pm}=(\phi_1^{\pm}, \phi_2^{\pm}, \cdots, \phi_n^{\pm}).
\]
Were the statement false, then $\Phi^+ > \Phi^-$ holds and \eqref{aa} implies that there exists $y\in (0,1)$ such that
\[
{a}(y)\leq \Phi ^{-}\leq \Phi ^{+}\leq {b}(y).
\]%
In particular, let $y_{1}\in (0,1)$ be the largest $y$ such that the above
is true, then $y_{1}$ is well defined. Without loss of generality, we assume
that
\[
\phi _{1}^{-}=a_{1}(y_{1})
\]%
with ${a}(y)=(a_{1}(y),a_{2}(y),\cdots ,a_{n}(y)).$

Due to the uniform boundedness of $\phi_1'$ and $\phi_1'',$  the fluctuation lemma of continuous functions implies that there exists $\{ \xi_m\}, m\in \mathbb{N},$ with $\lim_{m\to\infty} \xi_m =\infty$ such that
\[
\liminf_{\xi\to \infty}\phi_1(\xi)=\lim_{m\to\infty }\phi_1(\xi_m)=a_1(y_1)\le \limsup_{\xi\to \infty}\phi_1(\xi)\le b_1(y_1)
\]
and
\[
\liminf_{m\to \infty} (d_1 \phi_1''(\xi_m)-c \phi_1'(\xi_m)) \ge 0.
\]
At the same time, (H5) leads to
\[
\liminf_{m\to \infty } f_1^c(\Phi_{\xi_m}) >0,
\]
and we obtain a contradiction between
\begin{equation}\label{ee}
\liminf_{m\to \infty} (d_1 \phi_1''(\xi_m)-c \phi_1'(\xi_m)+ f_1^c(\Phi_{\xi_m})) >  0
\end{equation}
and
\[
\liminf_{m\to \infty} (d_1 \phi_1''(\xi_m)-c \phi_1'(\xi_m)+ f_1^c(\Phi_{\xi_m}))=0.
\]
The proof is complete.
\end{proof}

\begin{remark}{\rm
In fact, the proof of the existence of traveling wave solutions often implies the uniform boundedness of $\Phi'$  and $\Phi'',$ so we will not discuss the boundedness in the following two examples.}
\end{remark}

Now, we  recall two scalar equations to give a simple illustration of the theorem.
We first consider the example in Zou \cite{jcam}.
\begin{example}{\rm
In \eqref{f}, let
\[
f(u_t)=u(t-\tau)[1-u(t)].
\]
Then traveling wave solutions of the corresponding partial functional differential equation has been established by  Zou \cite{jcam}.
}
\end{example}

Define $\underline{E}=0, \overline{E}=1+k$ for any $k>0$ and
\[
a(s)=s, b(s)=(1+k)(1-s)+s, s\in [0,1].
\]
Then $[a(s), b(s)]$ is a contracting rectangle of the corresponding functional differential equation such that Theorem \ref{t1-cr} is applicable to the study of traveling wave solutions. Let $u(x,t)=\rho (x+ct)$ be a traveling wave solution of
\begin{equation*}
\frac{\partial u(x,t)}{\partial t}= \triangle u(x,t)+ ru(x,t-\tau)[1-u(x,t)], r>0.
\end{equation*}
If $\rho (\xi),\xi\in \mathbb{R},$ is bounded and $\liminf_{\xi\to\infty}\rho (\xi) >0,$ then $\lim_{\xi\to \infty }\rho (\xi)=1$ by Theorem \ref{t1-cr}.

If the quasimonotone condition does not hold,  Ma \cite{ma1} studied the traveling wave solutions by constructing proper auxiliary systems. In particular, the traveling wave solutions of the following example have been well studied.
\begin{example}
{\rm
Consider
\begin{equation}\label{ni}
\frac{\partial w(x,t)}{\partial t}= \triangle w(x,t)-w(x,t)+e^2w(x,t-\tau)e^{-w(x,t-\tau)}
\end{equation}
and denote $f(w)=e^2 we^{-w}.$
}
\end{example}

Using the results in Smith \cite[Section 5.2]{smith}, we can give two auxiliary quasimonotone equations to study the dynamics of the corresponding functional differential equations of \eqref{ni}. In addition, Ma \cite{ma1} presented two auxiliary quasimonotone equations of \eqref{ni} and obtained the convergence of traveling wave solutions. Let
\[
w(x,t)=\rho (x+ct)
\]
be a traveling wave solution of \eqref{ni} and denote
\[
\liminf_{\xi\to \infty}\rho(\xi)=\rho^-, \limsup_{\xi\to \infty}\rho(\xi)=\rho^+,
\]
then the auxiliary equations in Ma \cite{ma1} imply that
\begin{equation}\label{m111}
f(e)=e^{3-e}\le \rho^- \le \rho^+ \le  e
\end{equation}
and $f(w)$ is monotone decreasing for $w\in [e^{3-e}, e].$
If $\rho^+ = e,$ then a discussion similar to that of \eqref{ee} implies that  $f(\rho^-) \ge e$ by the monotonicity of $f,$ which is impossible by \eqref{m111}.  By the monotonicity of $f$ and \eqref{m111}, we further obtain that
\[
e^{3-e}< f(f(e^{3-e}))\le \rho^- \le \rho^+ \le  f(e^{3-e})=f(f(e))<e.
\]
Let $\underline{E}=f(f(e^{3-e})), \overline{E}=f(f(e)),$ then $[\widehat{\underline{E}},\widehat{\overline{E}}]$ defines an invariant region of the corresponding functional differential equation of \eqref{ni}.

To continue our discussion, define
\[
f^2(w)=f(f(w)),
\]
then $f^2(w)$ satisfies the following properties:
\begin{description}
\item[(F1)] $f^2(w)$ is monotone increasing for $w\in [e^{3-e}, e];$
\item[(F2)] $f^2 (w) >w, w\in [e^{3-e}, 2)$ while $f^2 (w) <w, w\in (2,e].$
\end{description}
For  convenience,  we give the graph of $f^2$ in Figure \ref{fig1}.

\begin{figure}
\begin{center}
\includegraphics[width=110mm]{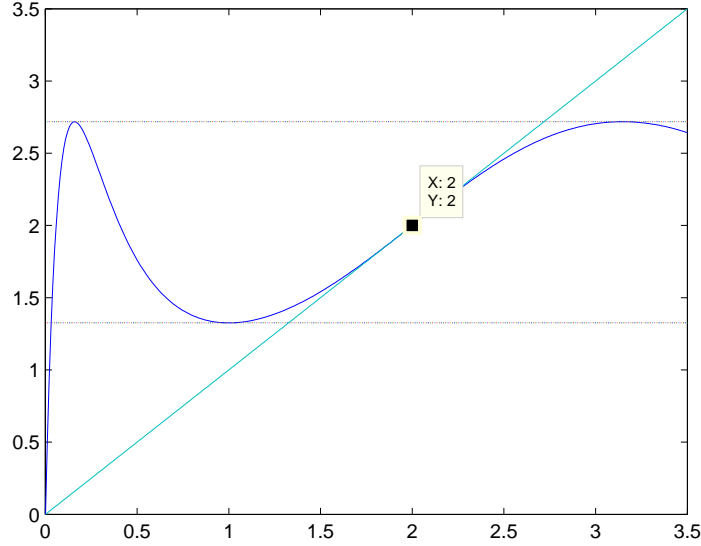}
\caption{The graph of the composition map $f^2$.}\label{fig1}
\end{center}
\end{figure}

Let $k=2, k_1= e^{3-e}(<2)$ and
\[
a(s)=sk+(1-s)k_{1}+\epsilon h(s), \,\,
b(s)=f(sk+(1-s)k_{1}),
\]
in which
\[
2h(s)=f^{2}(sk+(1-s)k_{1})-(sk+(1-s)k_{1})>0,
\]
and $\epsilon <1$ is small such that
\[
a(s)<2, s\in (0,1)
\]
and
\[
a(0)<f^2(e^{3-e})< f^2(e)=b(0).
\]
By (F2), $h(s)>0,s\in (0,1)$ and $h(s)=0,s=1.$

If $w(0)=a(s)$ with $a(s)\le w(-\tau)\le b(s),$ then
\begin{eqnarray*}
&&-w(0)+f(w(-\tau )) \\
&\geq &-a(s)+f(b(s)) \\
&=&-(sk+(1-s)k_{1}+\epsilon h(s))+f^{2}(sk+(1-s)k_{1}) \\
&=&(2- \epsilon)h(s)>0,s\in (0,1).
\end{eqnarray*}%
If $w(0)=b(s)$ with $a(s)\le w(-\tau)\le b(s),$ then%
\begin{eqnarray*}
&&-w(0)+f(w(-\tau )) \\
&\leq &-b(s)+f(a(s)) \\
&<&-b(s)+f(sk+(1-s)k_{1}) \\
&=&0,s\in (0,1),
\end{eqnarray*}
which implies that $[a(s),b(s)]$ satisfies (H5). Using Theorem \ref{t1-cr}, $\lim_{\xi\to \infty}\rho (\xi)=2$ holds.

\section{Generalized Upper and Lower Solutions}
\noindent

In this section, we shall study the existence of traveling wave solutions of delayed system \eqref{cfu} for any fixed $c>0$, where $f$ satisfies (H1)-(H3)  in Section 3. We first introduce the generalized upper and lower solutions of \eqref{fu} as follows.
\begin{definition}\label{ge}
{\rm
Assume that $\mathbb{T}\subset \mathbb{R}$ contains finite points of $\mathbb{R}.$ Then $ \overline{\Phi}=(\overline{\phi}_1, \overline{\phi}_2, \cdots, \overline{\phi}_n)\in X_{[0, \overline{E}]}$ and $ \underline{\Phi}=(\underline{\phi}_1, \underline{\phi}_2, \cdots, \underline{\phi}_n)\in X_{[0, \overline{E}]}$ are a pair of generalized upper and lower solutions of \eqref{fu} if for each $\xi\in \mathbb{R}\setminus \mathbb{T},$ $\overline{\Phi}''(\xi),\overline{\Phi}'(\xi), \underline{\Phi}''(\xi), \underline{\Phi}'(\xi)$ are bounded and continuous such that
\begin{equation}\label{geu}
d_i\overline{\phi}''_i(\xi)-c\overline{\phi}'_i(\xi)+f_i^c(\widetilde{\Phi}_{\xi})\le 0
\end{equation}
with $\widetilde{\Phi}(\xi)=(\widetilde{\phi}_1(\xi), \widetilde{\phi}_2(\xi), \cdots, \widetilde{\phi}_n(\xi))\in X_{[0, \overline{E}]}$ satisfying
\[
\underline{\phi}_j(\xi+cs)\le \widetilde{\phi}_j(\xi+cs)\le \overline{\phi}_j(\xi+cs), \widetilde{\phi}_i(\xi)= \overline{\phi}_i(\xi),s\in [-\tau, 0], i,j \in I,
\]
and
\begin{equation}\label{gel}
d_i\underline{\phi}''_i(\xi)-c\underline{\phi}'_i(\xi)+f_i^c(\widehat{\Phi}_{\xi})\ge 0
\end{equation}
with $\widehat{\Phi}(\xi)=(\widehat{\phi}_1(\xi), \widehat{\phi}_2(\xi), \cdots, \widehat{\phi}_n(\xi))\in X_{[0, \overline{E}]}$ satisfying
\[
\underline{\phi}_j(\xi+cs)\le \widehat{\phi}_j(\xi+cs)\le \overline{\phi}_j(\xi+cs), \widehat{\phi}_i(\xi)= \underline{\phi}_i(\xi),s\in [-\tau, 0], i,j \in I.
\]
}
\end{definition}

\begin{remark}{\rm
If $f$ is quasimonotone, then Definition \ref{ge} is equivalent to Ma \cite[Definition 2.2]{ma01}, Wu and Zou \cite[Definition 3.2]{wuzou2}; if $f$ is mixed quasimonotone, then Definition \ref{ge} becomes Lin et al. \cite[Definition 3.1]{llm}.}
\end{remark}

For any fixed $\xi\in \mathbb{R},$ let $\beta >0$ be a fixed constant such that
\[
\beta \phi_i(\xi) +f_i^c (\Phi_{\xi})
\]
is monotone increasing in $\phi_i(\xi), i\in I, \Phi=(\phi_1,\phi_2, \cdots, \phi_n)\in X_{[0, \overline{E}]}.$ From (H3), $\beta$ is well defined.

Define
constants
\[
\nu _{i1}(c)=\frac{c-\sqrt{c^{2}+4\beta d_{i}}}{2d_{i}},
\,\,\nu _{i2}(c)=\frac{c+%
\sqrt{c^{2}+4\beta d_{i}}}{2d_{i}},i\in I.
\]%
For the sake of simplicity, we  denote $\nu _{i1}=\nu
_{i1}(c),\nu _{i2}=\nu _{i2}(c)$ without confusion. Then $\nu
_{i1}<0<\nu _{i2}$ and
\[
d_i\nu _{ij}^{2}-c\nu _{ij}-\beta =0,\,\, i\in I,j=1,2.
\]%
For $\Phi =(\phi _{1},\phi _{2},\cdots ,\phi _{n})\in X_{[0, \overline{E}]},$ define $%
F=(F_{1},F_{2},\cdots ,F_{n}):X\rightarrow X$ by
\begin{equation}
F_{i}(\Phi )(\xi )=\frac{1}{d_{i}(\nu _{i2}-\nu _{i1})}\left[
\int_{-\infty
}^{\xi }e^{\nu _{i1}(\xi -s)}+\int_{\xi }^{+\infty }
e^{\nu _{i2}(\xi -s)}%
\right] L_{i}(\Phi )(s)ds,  \label{5.3}
\end{equation}%
herein $L(\Phi )(s)=(L_{1}(\Phi )(s),L_{2}(\Phi )(s),\cdots
,L_{n}(\Phi )(s)) $ is formulated by
\[
L_{i}(\Phi )(\xi )=\beta \phi _{i}(\xi )+f_i^c(\Phi_{\xi}) ,i\in I.
\]
Now, to prove the existence of \eqref{fu}, it is sufficient to seek after a fixed point of $F$ (see Wu and Zou \cite{wuzou2}).

Let $\sigma <\min_{i\in I}\{- \nu _{i1}\}$ be a positive constant and $|\cdot|$ denote the supremum norm
in $\mathbb{R}^n.$ Define
\[
B_{\sigma }\left( \mathbb{R},\mathbb{R}^{n}\right) =\left\{ \Phi(x) \in
X:\sup_{x \in \mathbb{R}}\left\vert \Phi (x)\right\vert
e^{-\sigma \left\vert x \right\vert }<\infty \right\}
\]%
and
\[
\left\vert \Phi \right\vert _{\sigma }=\sup_{x \in
\mathbb{R}}\left\vert \Phi (x )\right\vert e^{-\sigma \left\vert
x \right\vert }.
\]%
Then it is easy to check that $ B_{\sigma }\left( \mathbb{R},\mathbb{R}%
^{n}\right) $ is a
Banach space with the decay norm $\left\vert \cdot \right\vert _{\sigma }$.

Before giving our main conclusion of this section, we first present some calculations. If $\phi (s)$ is twice differentiable, then
\begin{eqnarray}
&&\frac{d}{d s}\left[ e^{-\nu _{ij}s}\left( (c-d_{i}\nu _{ij})\phi
(s)-d_{i}\phi ^{\prime }(s)\right) \right]   \nonumber \\
&=&e^{-\nu _{ij}s}\left[ -\nu _{ij}(c-d_{i}\nu _{ij})\phi (s)+d_{i}\nu
_{ij}\phi ^{\prime }(s)\right]   \nonumber \\
&&+e^{-\nu _{ij}s}\left( (c-d_{i}\nu _{ij})\phi ^{\prime }(s)-d_{i}\phi
^{\prime \prime }(s)\right)   \nonumber \\
&=&e^{-\nu _{ij}s}\left[ -d_{i}\phi ^{\prime \prime }(s)+c\phi ^{\prime
}(s)+\beta \phi (s)\right]   \label{aux-1}
\end{eqnarray}%
because of $d_{i}\nu _{ij}^{2}-c\nu _{ij}-\beta =0,i\in I,j=1,2.$

If a bounded function $\phi (s)$ admits continuous and bounded derivatives
$\phi ^{\prime }(s)$ and $\phi ^{\prime \prime }(s)$\ for $s\in (a,b)$ with $a<b,$ then
\begin{eqnarray}
&&\int_{a}^{b}e^{-\nu _{ij}s}\left[ -d_{i}\phi ^{\prime \prime }(s)+c\phi
^{\prime }(s)+\beta \phi (s)\right] d s  \nonumber \\
&=&e^{-\nu _{ij}b}\left( (c-d_{i}\nu _{ij})\phi (b-)-d_{i}\phi ^{\prime
}(b-)\right)   \nonumber \\
&&-e^{-\nu _{ij}a}\left( (c-d_{i}\nu _{ij})\phi (a+)-d_{i}\phi ^{\prime
}(a+)\right) .  \label{aux-2}
\end{eqnarray}%
Moreover, if $\xi \in (a,b),$ then
\begin{eqnarray}
&&\frac{1}{d_{i}(\nu _{i2}-\nu _{i1})}\left[ \int_{a}^{\xi }e^{\nu _{i1}(\xi
-s)}+\int_{\xi }^{b}e^{\nu _{i2}(\xi -s)}\right] (\beta \phi _{i}(s)+c\phi
_{i}^{\prime }(s)-d_{i}\phi _{i}^{\prime \prime }(s))d s  \nonumber \\
&=&\left. \frac{1}{d_{i}(\nu _{i2}-\nu _{i1})}\left[ e^{\nu _{i1}(\xi
-s)}\left( (c-d_{i}\nu _{i1})\phi (s)-d_{i}\phi ^{\prime }(s)\right) \right]
\right\vert _{a}^{\xi }  \nonumber \\
&&+\left. \frac{1}{d_{i}(\nu _{i2}-\nu _{i1})}\left[ e^{\nu _{i2}(\xi
-s)}\left( (c-d_{i}\nu _{i2})\phi (s)-d_{i}\phi ^{\prime }(s)\right) \right]
\right\vert _{\xi }^{b}  \nonumber \\
&=&\underline{\phi }_{i}(\xi )+\frac{e^{-\nu _{i2}b}\left( (c-d_{i}\nu
_{i2})\phi (b-)-d_{i}\phi ^{\prime }(b-)\right) }{d_{i}(\nu _{i2}-\nu _{i1})}
\nonumber \\
&&-\frac{e^{-\nu _{i1}a}\left( (c-d_{i}\nu _{i1})\phi (a+)-d_{i}\phi
^{\prime }(a+)\right) }{d_{i}(\nu _{i2}-\nu _{i1})}.  \label{aux-3}
\end{eqnarray}

Now we state and prove the main result of this section.

\begin{theorem}\label{th5.1}
Assume that $ \overline{\Phi}=(\overline{\phi}_1, \overline{\phi}_2, \cdots, \overline{\phi}_n)\in X_{[0, \overline{E}]}$ and
$\underline{\Phi}=(\underline{\phi}_1, \underline{\phi}_2, \cdots, \underline{\phi}_n)\in X_{[0, \overline{E}]}$ are a pair of generalized upper and lower solutions of \eqref{fu} such that
\[
\overline{\Phi}(\xi)\ge \overline{\Phi}(\xi),\xi\in \mathbb{R}
\]
and
\[
\overline{\phi}'_i(\xi+)\le \overline{\phi}'_i(\xi-), \underline{\phi}'_i(\xi+)\ge \underline{\phi}'_i(\xi-), \xi \in \mathbb{T}, i\in I.
\]
Then \eqref{fu} has a solution $\Phi$ such that $\underline{\Phi}\le {\Phi}\le \overline{\Phi}.$
\end{theorem}
\begin{proof}
Let
\[
\Gamma=\{\Phi(\xi)\in X: \underline{\Phi}(\xi)\le \Phi(\xi)\le
\overline{\Phi}(\xi)\}.
\]
It is clear that $\Gamma $ is nonempty and convex. Moreover, it is
closed and bounded with respect to the decay norm
$|\cdot|_{\sigma}.$
Choose $\Phi=(\phi_1,\phi_2, \cdots, \phi_n) \in \Gamma,$ then for each fixed $\xi\in\mathbb{R}$, the definition of $\beta$ implies that
\begin{equation}\label{ull}
\beta \underline{\phi}_i(\xi)+f_i^c(\widehat{\Phi}_{\xi})\le \beta \phi_i(\xi)+f_i^c (\Phi_{\xi}) \le \beta \overline{\phi}_i(\xi)+f_i^c(\widetilde{\Phi}_{\xi})
\end{equation}
with $\widetilde{\Phi}(\xi+cs)=(\widetilde{\phi}_1(x+cs), \widetilde{\phi}_2(x+cs), \cdots, \widetilde{\phi}_n(x+cs))$ satisfying
\[
\begin{cases}
\widetilde{\phi}_j(\xi+cs)=\phi_j(\xi+cs), j\neq i ,
 s\in [-\tau, 0],\\ \widetilde{\phi}_i(\xi+cs)=\phi_i(\xi +cs), s\in [-\tau, 0),
  \widetilde{\phi}_i(\xi)=\overline{\phi}_i(\xi)
\end{cases}
\]
and  $\widehat{\Phi}(\xi+cs)=(\widehat{\phi}_1(x+cs), \widehat{\phi}_2(x+cs), \cdots, \widehat{\phi}_n(x+cs))$ satisfying
\[
\begin{cases}
\widehat{\phi}_j(\xi+cs)=\phi_j(\xi+cs), j\neq i ,
 s\in [-\tau, 0], \\
 \widehat{\phi}_i(\xi+cs)=\phi_i(\xi +cs), s\in [-\tau, 0), \widehat{\phi}_i(\xi)=\underline{\phi}_i(\xi).
\end{cases}
\]

By \eqref{5.3}-\eqref{ull}, we obtain
\begin{eqnarray}
&&F_{i}(\Phi )(\xi )  \nonumber \\
&\geq &\frac{1}{d_{i}(\nu _{i2}-\nu _{i1})}\left[ \int_{-\infty }^{\xi
}e^{\nu _{i1}(\xi -s)}+\int_{\xi }^{+\infty }e^{\nu _{i2}(\xi -s)}\right]
(\beta \underline{\phi }_{i}(s)+c\underline{\phi }_{i}^{\prime }(s)-d_{i}%
\underline{\phi }_{i}^{\prime \prime }(s))ds  \nonumber \\
&=&\underline{\phi }_{i}(\xi )+\sum_{T_{j}\in \mathbb{T}}\frac{\min \left\{ e^{\nu _{i1}(\xi
-T_{j})},e^{\nu _{i2}(\xi -T_{j})}\right\} }{\nu _{i2}-\nu _{i1}}\left[
\underline{\phi }_{i}^{\prime }(T_{j}+)-\underline{\phi }_{i}^{\prime
}(T_{j}-)\right]   \nonumber \\
&\geq &\underline{\phi }_{i}(\xi ), \xi \in \mathbb{R} \setminus \mathbb{T}.  \label{any}
\end{eqnarray}
Using the continuity of $F_{i}(\Phi )(\xi ), \underline{\phi }_{i}(\xi ),$ we obtain
\[
F_{i}(\Phi )(\xi )\ge  \underline{\phi }_{i}(\xi ), \xi\in \mathbb{R}.
\]

In a similar way, we have
\[
F_{i}(\Phi )(\xi ) \le \overline{\phi}_i(\xi), i\in I, \xi\in \mathbb{R},
\]
and
\[
F: \Gamma \to \Gamma.
\]

Moreover, similar to those in Huang and Zou \cite{huangzou2}, Li et al. \cite[Lemma 3.6]{llr} and Ma \cite{ma01}, $F: \Gamma \to \Gamma$ is completely continuous in the sense of the decay norm $|\cdot|_{\sigma}$. Using Schauder's fixed point theorem, we complete the proof.
\end{proof}
\begin{remark}\label{stri}
{\rm
Let $\Phi$ be a solution given by Theorem \ref{th5.1}.  From \eqref{any}, we obtain
\[
\phi_i(\xi)>\underline{\phi }_{i}(\xi ),\xi\in\mathbb{R}, i\in I
\]
if one of the following statements is true: 1) for each $i\in I,$ $\underline{\phi}'_i(\xi+)> \underline{\phi}'_i(\xi-)$ for some $\xi\in \mathbb{T}$; 2) for each $i\in I,$ \eqref{gel} is strict on an nonempty interval. Similarly, we have
\[
\phi_i(\xi)<\overline{\phi }_{i}(\xi ),\xi\in\mathbb{R}, i\in I
 \]
if one of the following statements is true:  1) for each $i\in I,$ $\overline{\phi}'_i(\xi+)< \overline{\phi}'_i(\xi-)$ for some $\xi\in \mathbb{T}$; 2) for each $i\in I,$ \eqref{geu} is strict on an nonempty interval. }
\end{remark}
\begin{remark}
{\rm
Let $\Phi$ be a solution given by Theorem \ref{th5.1}. Since $F(\Phi)(\xi)=\Phi (\xi),$ then $\Phi '(\xi)$ and $\Phi ''(\xi)$ are uniformly bounded by the bounds of $\Phi(\xi)$.}
\end{remark}

\section{Traveling Wave Solutions of the Lotka-Volterra System \eqref{1}}
\noindent

In this section, we study the existence and nonexistence of traveling wave solutions of \eqref{1} with $\Omega =\mathbb{R}, a_i>0,i\in I$. We first introduce some notations.

Let $u(x,t)=\Psi (x+ct)=(\psi_1(x+ct),\psi_2(x+ct),\cdots,
\psi_n(x+ct) )$ be a traveling wave solution of \eqref{1}. Then
$\Psi(\xi)$ satisfies the following functional differential system
\begin{equation}
d_{i}\psi_i ^{\prime \prime }(\xi )-c\psi _{i}^{\prime }(\xi
)+r_{i}\psi
_{i}(\xi )\left[ 1-\sum_{j=1}^{n}c_{ij}\int_{-\tau }^{0}
\psi _{j}(\xi +cs)d%
\overline{\eta }_{ij}(s)\right]=0 ,i\in I,\xi\in\mathbb{R}.  \label{5.0}
\end{equation}
Similar to those in Li et al. \cite{llr}, Lin et al. \cite{llm}, Martin and Smith \cite{martin2}, we  consider the
invasion waves of all competitors which satisfy the following asymptotic boundary conditions
\begin{equation}
\lim_{\xi \rightarrow -\infty }\psi _{i}(\xi )=0,\,\,\lim_{\xi
\rightarrow \infty }\psi _{i}(\xi )=u_{i}^{\ast }.  \label{5.1}
\end{equation}

\begin{remark}{\rm
From the viewpoint of population dynamics, \eqref{5.0}-\eqref{5.1} formulate the synchronous invasion of all competitors, we refer to Shigesada and Kawasaki \cite[Chapter 7]{shi} for  the historical records of the expansion of the geographic range of several plants in North American after the last ice age (16,000 years ago).}
\end{remark}

\subsection{Existence of Traveling Wave Solutions}
\noindent

To construct upper and lower solutions, we define some constants. For any fixed $c> \max_{i\in I}\{2\sqrt{d_ir_i}\},$ define constants $\gamma
_{i1}$ and $\gamma _{i2}$ such that $0<\gamma _{i1}<\gamma _{i2}$ and
\begin{equation}
d_i\gamma _{i1}^2-c\gamma _{i1}+r_i=d_i\gamma _{i2}^2-c\gamma _{i2}+r_i=0%
\text{ for }i\in I.\label{gam}
\end{equation}
Assume that  $q>1$ holds and  $\eta$ satisfies
\begin{equation}
\eta \in \left( 1,\min_{ i,j\in I}\left\{ \frac{\gamma _{i2}}{\gamma
_{i1}},\frac{\gamma _{i1}+\gamma _{j1}}{\gamma _{i1}}\right\}
\right) . \label{eta}
\end{equation}
Define continuous functions $\underline{\psi}_i(\xi)$ and
$\overline{\psi}_i(\xi)$ as follows
\[
\underline{\psi}_i(\xi)=\max\{e ^{\gamma _{i1}\xi}-qe ^{\eta\gamma
_{i1}\xi}, 0\},\,\, \overline{\psi}_i(\xi)=\min\{e ^{\gamma _{i1}\xi},
(a_ic_{ii})^{-1}\}, i\in I.
\]

\begin{lemma}\label{le5.1}
If $q>1$ is large, then $\underline{\psi}_i(\xi)$ and
$\overline{\psi}_i(\xi)$  are generalized upper and lower solutions of \eqref{5.0}.
\end{lemma}
\begin{proof}
By the monotonicity, it suffices to prove that%
\begin{equation}
d_{i}\overline{\psi }_{i}^{\prime \prime }(\xi )-c\overline{\psi }%
_{i}^{\prime }(\xi )+r_{i}\overline{\psi }_{i}(\xi )\left[ 1-c_{ii}a_{i}%
\overline{\psi }_{i}(\xi )-\sum_{j=1}^{n}c_{ij}\int_{-\tau }^{0}\underline{%
\psi }_{j}(\xi +cs)d\eta _{ij}(s)\right] \leq 0,i\in I,  \label{5u}
\end{equation}%
and%
\begin{equation}
d_{i}\underline{\psi }_{i}^{\prime \prime }(\xi )-c\underline{\psi }%
_{i}^{\prime }(\xi )+r_{i}\underline{\psi }_{i}(\xi )\left[ 1-c_{ii}a_{i}%
\underline{\psi }_{i}(\xi )-\sum_{j=1}^{n}c_{ij}\int_{-\tau }^{0}\overline{%
\psi }_{j}(\xi +cs)d\eta _{ij}(s)\right] \geq 0,i\in I  \label{5l}
\end{equation}%
if $\underline{\psi }_{i}(\xi )$ and $\overline{\psi }_{i}(\xi ),i\in I,$
are differentiable.

If $\overline{\psi }_{i}(\xi )=e^{\gamma _{i1}\xi }<(a_{i}c_{ii})^{-1},$ then%
\begin{eqnarray*}
&&d_{i}\overline{\psi }_{i}^{\prime \prime }(\xi )-c\overline{\psi }%
_{i}^{\prime }(\xi )+r_{i}\overline{\psi }_{i}(\xi )\left[ 1-c_{ii}a_{i}%
\overline{\psi }_{i}(\xi )-\sum_{j=1}^{n}c_{ij}\int_{-\tau }^{0}\underline{%
\psi }_{j}(\xi +cs)d\eta _{ij}(s)\right]  \\
&\leq &d_{i}\overline{\psi }_{i}^{\prime \prime }(\xi )-c\overline{\psi }%
_{i}^{\prime }(\xi )+r_{i}\overline{\psi }_{i}(\xi ) \\
&=&e^{\gamma _{i1}\xi }[d_{i}\gamma _{i1}^{2}-c\gamma _{i1}+r_{i}] \\
&=&0.
\end{eqnarray*}
If $\overline{\psi }_{i}(\xi )=(a_{i}c_{ii})^{-1}<e^{\gamma _{i1}\xi },$ then%
\begin{eqnarray*}
&&d_{i}\overline{\psi }_{i}^{\prime \prime }(\xi )-c\overline{\psi }%
_{i}^{\prime }(\xi )+r_{i}\overline{\psi }_{i}(\xi )\left[ 1-c_{ii}a_{i}%
\overline{\psi }_{i}(\xi )-\sum_{j=1}^{n}c_{ij}\int_{-\tau }^{0}\underline{%
\psi }_{j}(\xi +cs)d\eta _{ij}(s)\right]  \\
&\leq &d_{i}\overline{\psi }_{i}^{\prime \prime }(\xi )-c\overline{\psi }%
_{i}^{\prime }(\xi )+r_{i}\overline{\psi }_{i}(\xi )\left[ 1-c_{ii}a_{i}%
\overline{\psi }_{i}(\xi )\right]  \\
&=&0,
\end{eqnarray*}%
and this completes the proof of \eqref{5u}.

If $\underline{\psi }_{i}(\xi )=0>e^{\gamma _{i1}\xi }-qe^{\eta \gamma
_{i1}\xi },$ then%
\[
d_{i}\underline{\psi }_{i}^{\prime \prime }(\xi )-c\underline{\psi }%
_{i}^{\prime }(\xi )+r_{i}\underline{\psi }_{i}(\xi )\left[ 1-c_{ii}a_{i}%
\underline{\psi }_{i}(\xi )-\sum_{j=1}^{n}c_{ij}\int_{-\tau }^{0}\overline{%
\psi }_{j}(\xi +cs)d\eta _{ij}(s)\right] =0.
\]
If $\underline{\psi }_{i}(\xi )=e^{\gamma _{i1}\xi }-qe^{\eta \gamma
_{i1}\xi }>0,$ then%
\begin{eqnarray*}
&&d_{i}\underline{\psi }_{i}^{\prime \prime }(\xi )-c\underline{\psi }%
_{i}^{\prime }(\xi )+r_{i}\underline{\psi }_{i}(\xi )\left[ 1-c_{ii}a_{i}%
\underline{\psi }_{i}(\xi )-\sum_{j=1}^{n}c_{ij}\int_{-\tau }^{0}\overline{%
\psi }_{j}(\xi +cs)d\eta _{ij}(s)\right]  \\
&=&d_{i}\underline{\psi }_{i}^{\prime \prime }(\xi )-c\underline{\psi }%
_{i}^{\prime }(\xi )+r_{i}\underline{\psi }_{i}(\xi ) \\
&&-r_{i}c_{ii}a_{i}\underline{\psi }_{i}^{2}(\xi )-r_{i}\underline{\psi }%
_{i}(\xi )\sum_{j=1}^{n}c_{ij}\int_{-\tau }^{0}\overline{\psi }_{j}(\xi
+cs)d\eta _{ij}(s) \\
&=&-qe^{\eta \gamma _{i1}\xi }[d_{i}\eta ^{2}\gamma _{i1}^{2}-c\eta \gamma
_{i1}+r_{i}] \\
&&-r_{i}c_{ii}a_{i}\underline{\psi }_{i}^{2}(\xi )-r_{i}\underline{\psi }%
_{i}(\xi )\sum_{j=1}^{n}c_{ij}\int_{-\tau }^{0}\overline{\psi }_{j}(\xi
+cs)d\eta _{ij}(s),
\end{eqnarray*}%
and the monotonicity of $\overline{\psi }_{j}$ indicates that
\begin{eqnarray*}
&&-r_{i}c_{ii}a_{i}\underline{\psi }_{i}^{2}(\xi )-r_{i}\underline{\psi }%
_{i}(\xi )\sum_{j=1}^{n}c_{ij}\int_{-\tau }^{0}\overline{\psi }_{j}(\xi
+cs)d\eta _{ij}(s) \\
&\geq &-r_{i}c_{ii}a_{i}\underline{\psi }_{i}^{2}(\xi )-r_{i}\underline{\psi
}_{i}(\xi )\sum_{j=1}^{n}c_{ij}\overline{\psi }_{j}(\xi ) \\
&\geq &-r_{i}c_{ii}a_{i}e^{2\gamma _{i1}\xi
}-r_{i}\sum_{j=1}^{n}c_{ij}e^{(\gamma _{i1}+\gamma _{j1})\xi }.
\end{eqnarray*}%
Therefore, we only need to verify that%
\begin{equation}
-qe^{\eta \gamma _{i1}\xi }[d_{i}\eta ^{2}\gamma _{i1}^{2}-c\eta \gamma
_{i1}+r_{i}]-r_{i}c_{ii}a_{i}e^{2\gamma _{i1}\xi
}-r_{i}\sum_{j=1}^{n}c_{ij}e^{(\gamma _{i1}+\gamma _{j1})\xi }\geq 0.
\label{5lo}
\end{equation}%
Since $q>1,$ we have $\xi <0$ and%
\[
e^{\eta \gamma _{i1}\xi }>e^{2\gamma _{i1}\xi },e^{\eta \gamma _{i1}\xi
}>e^{(\gamma _{i1}+\gamma _{j1})\xi },
\]%
which imply that \eqref{5lo} is true if
\[
q>\frac{-r_{i}c_{ii}a_{i}-r_{i}\sum_{j=1}^{n}c_{ij}}{d_{i}\eta ^{2}\gamma
_{i1}^{2}-c\eta \gamma _{i1}+r_{i}}+1>1.
\]%
Let%
\[
q=\max_{i\in I}\left\{ \frac{-r_{i}c_{ii}a_{i}-r_{i}\sum_{j=1}^{n}c_{ij}}{%
d_{i}\eta ^{2}\gamma _{i1}^{2}-c\eta \gamma _{i1}+r_{i}}\right\} +2,
\]%
then \eqref{5l} is true and we complete the proof.
\end{proof}

From Lemma \ref{le5.1}, Theorem \ref{th5.1} and Remark \ref{stri}, we obtain the following
result.
\begin{theorem}\label{th5.0}
For each $c> \max_{i\in
I}\{2\sqrt{d_ir_i}\},$ \eqref{5.0} has a strictly positive solution
$\Psi=(\psi_1,\psi_2,\cdots,\psi_n)$ such that
\begin{equation}
\lim_{\xi
\rightarrow -\infty }\psi _{i}(\xi )e^{-\gamma _{i1}\xi
}=1,\,\, \underline{\psi}_i (\xi) <  \psi_i (\xi)<  \overline{\psi}_i (\xi), \xi \in \mathbb{R}, i\in I.  \label{5.}
\end{equation}
\end{theorem}

\subsection{Asymptotic Behavior of Traveling Wave Solutions}
\noindent

The following is the main conclusion of this subsection.

\begin{theorem}\label{th5}
Assume that \eqref{2} holds. If $\Psi(\xi)$ is formulated by Theorem \ref{th5.0}, then  \eqref{5.1} is true.
\end{theorem}
\begin{proof}
Note
that $\Psi(x+ct)$ is a special classical solution of the following initial
value problem
\begin{equation}\label{5.2}
\begin{cases}
\dfrac{\partial u_{i}(x,t)}{\partial t}=d_{i}\Delta
u_{i}(x,t)+r_{i}u_{i}(x,t)\left[ 1-\sum_{j=1}^{n}c_{ij}\int_{-\tau
}^{0}u_{j}(x,t+s)d\overline{\eta }_{ij}(s)\right] ,\\
u_{i}(x,s)=\psi _{i}(x+cs), \end{cases}
\end{equation}
where $x\in\mathbb{R}, t>0,s\in [-\tau, 0].$
The boundedness and smoothness of $\Psi(x+ct)$ imply that $\psi_i(x+ct)$ is an
upper solution to the following Fisher equation
\[
\dfrac{\partial u_{i}(x,t)}{\partial t}=d_{i}\Delta
u_{i}(x,t)+r_{i}u_{i}(x,t)\left[ 2- \sum_{j=1}^n c_{ij}
(c_{jj}a_j)^{-1}-a_ic_{ii}u_{i}(x,t)\right].
\]
Thus Lemma \ref{le1.1} asserts that
\[
\liminf_{\xi\to\infty}\psi_i(\xi)\ge 2- \sum_{j=1}^n c_{ij}
(c_{jj}a_j)^{-1}>0, i\in I.
\]

Denote
\[
\liminf_{\xi\to\infty}\psi_i(\xi)=\psi_i^-,\,\,\,
\limsup_{\xi\to\infty}\psi_i(\xi)=\psi_i^+,
\]
then there exists $s\in (0,1]$ such that
\[
a_i(s)\le \psi_i^- \le \psi_i^+ \le b_i(s)
\]
with
\[
a_i(s)=su_i^*, b_i(s)= su_i^*+(1-s)[(c_{ii}a_i)^{-1}+\epsilon ],\epsilon >0,
\]
\[
a(s)=(a_1(s),a_2(s),\cdots, a_n(s)), \,\, b(s)=(b_1(s),b_2(s),\cdots, b_n(s)).
\]
By Smith \cite[Lemma 7.4]{smith}, there exists a constant $\epsilon >0$ such that $[{a}(s), {b}(s)]$ defines a strictly contracting rectangle of the corresponding functional differential equations of \eqref{1}. Applying Theorem \ref{t1-cr}, we complete the proof.
\end{proof}

Furthermore, from the proof of Theorem \ref{th5}, we obtain
the following result.

\begin{theorem}\label{th8}
Assume that $\Psi(\xi)=(\psi_1(\xi),\psi_2(\xi),\cdots,\psi_n(\xi))$
is a strictly positive solution to \eqref{5.0} and satisfies
\[
0\le \psi_i(\xi)< (a_ic_{ii})^{-1}, i\in I, \xi\in\mathbb{R}.
\]
Then
 $\lim_{\xi\to \infty}\psi _i (\xi)=u_i^*$ if \eqref{2} holds.
\end{theorem}

\subsection{Nonexistence of Traveling Wave Solutions}

\begin{theorem}\label{no}
Assume that $2\sqrt{d_{i_0}r_{i_0}}= \max_{i\in I}\{2\sqrt{d_ir_i}\}$ for some $i_0 \in I.$ If $c<2\sqrt{d_{i_0}r_{i_0}},$ then \eqref{5.0}  does not have a
bounded positive solution $\Psi=(\psi_1,\psi_2,\cdots,\psi_n)$ satisfying
\begin{equation}\label{as}
\lim_{\xi\to-\infty}\psi_i(\xi)=0,
\liminf_{\xi\to\infty}\psi_{i_0}(\xi)>0,
\xi\in\mathbb{R}, i\in I.
\end{equation}
Moreover, if \eqref{2} holds,  then \eqref{5.0} does not have a strictly  positive
solution such that
\[
0\le \psi_i(\xi)< (a_ic_{ii})^{-1}, i\in I, \xi\in\mathbb{R}.
\]
\end{theorem}
\begin{proof}
Without loss of generality, we assume that $2\sqrt{d_1r_1}= \max_{i\in I}\{2\sqrt{d_ir_i}\}.$
If \eqref{2} holds, then \eqref{as} is obtained by Theorem \ref{th8}.
So we suppose that \eqref{as} is true. Were the statement false, then there exists some $c'\in (0,
2\sqrt{d_1r_1})$ such that \eqref{5.0} with $c=c'$ has a positive
solution $\Psi=(\psi_1,\psi_2,\cdots,\psi_n)$. Then \eqref{as}
implies that there exists $M>0$ such that
$\psi_1(\xi)=\psi_1(x+c't)$ satisfies
\begin{equation*}
\begin{cases}
\frac{\partial w(x,t)}{\partial t}\ge d_1 \Delta w(x,t)+r'_1w
(x,t)\left[
1-Mw (x,t)\right] ,\\
w(x,0)=\psi_1(x),
\end{cases}
\end{equation*}
where $4\sqrt{d_1r_1'}=2\sqrt{d_1r_1}+c'.$ In fact, if $\xi\to -
\infty,$ then $\lim_{\xi\to-\infty}\psi_i(\xi)=0$ ensures the
admissibility of $r_1'.$ Otherwise,
$\liminf_{\xi\to\infty}\psi_1(\xi)>0$ implies
the admissibility of $M$ (may be large but finite). By the
smoothness of $\psi_1(x+c't),$ we know that $\psi_1(x+c't)$ is an
upper solution to the following initial value problem
\begin{equation}\label{fush}
\begin{cases}
\frac{\partial w(x,t)}{\partial t}= d_1 \Delta w(x,t)+r'_1w
(x,t)\left[
1-Mw (x,t)\right] ,\\
w(x,0)=\psi_1(x).
\end{cases}
\end{equation}

Let $-2x= (2 \sqrt{d_1r_1'}+c')t,$
then $t\to\infty$ implies that $x+c't\to -\infty$ such that $u_1(x,t)=\psi_1(x+c't)\to 0.$
At the same time,
\[
-2x= \left(2 \sqrt{d_1r_1'}+c'\right)t< 4\sqrt{d_1r_1'}t
\]
and Lemma \ref{le1.1} lead to $\liminf_{t\to\infty}u_1(x,t) \ge 1/M$ in \eqref{fush}, a contradiction occurs.
The proof
is complete.
\end{proof}

\begin{remark}{\rm
Even if $c<  \max_{i\in I}\{2\sqrt{d_ir_i}\},$ \eqref{5.0} may have
a nontrivial positive solution $\Psi=(\psi_1,\psi_2,\cdots,\psi_n)$
such that $\max_{i\in I}\liminf_{\xi\to -\infty}\psi_i(\xi)>0.$ We
refer to Guo and Liang \cite{gl}, Huang \cite{huang} for some recent results of the corresponding undelayed equations.}
\end{remark}

\begin{remark}{\rm
Theorem \ref{no} completes the discussions of Li et al. \cite[Example 5.1]{llr} and Lin et al. \cite[Example 5.2]{llm} by presenting the nonexistence of traveling wave solutions. Therefore, it also improves some results for undelayed systems by confirming the nonexistence of traveling wave solutions without the requirement of monotonicity, for example, Ahmad and  Lazer \cite{ah} and Tang and Fife \cite{ta}. See next subsection.}
\end{remark}

\subsection{Existence of Nonmonotone Traveling Wave Solutions}
\noindent

If $\tau =0,$   Ahmad and Lazer \cite{ah}, Tang and Fife \cite{ta} proved the existence of monotone traveling wave solutions of \eqref{1}. Recently, Fang and Wu \cite{fang} also confirmed the existence of monotone traveling wave solutions if $\tau $ is small and $n=2$. In particular, Tang and Fife \cite{ta} thought that the monotonicity was a technical requirement and conjectured the existence of nonmonotone traveling wave solutions if $\tau =0$ and $n=2.$

In the previous section, we obtained the existence of traveling wave solutions connecting $0$ with $u^*.$ Because our requirement for the auxiliary functions was very weak,  we can present some sufficient conditions of the existence of nonmonotone traveling wave solutions if $a_i=1$ for all $i\in I$.

By rescaling, it suffices to consider
\eqref{1} with $c_{ii}=1, i\in I.$ Then \eqref{2} implies that $c_{ij}<1, i\neq j, i,j \in I,$ which further indicates that we can obtain a \emph{fixed} $q$ such that Lemma \ref{le5.1} holds for \emph{any} fixed $c$ (so $\eta$ can be a constant) and for \emph{all} $c_{ij}<1, i\neq j, i,j \in I.$

Therefore, for each fixed $c,$ there exist $m_i >0$ (e.g., $m_i=\sup_{\xi\in\mathbb{R}}\underline{\psi}_i(\xi)$) independent of $c_{ij}$ such that
\[
\sup_{\xi\in\mathbb{R}}\underline{\psi}_i(\xi) \ge  m_i, i\in I.
\]
Let $u^*$ be the function of $c_{ij}$ with $i\ne j, i,j\in I,$ then there exist $c_{ij}$ with $i\ne j, i,j\in I,$ such that \eqref{2}  is true and
\[
u_i^* \le m_i
\]
holds for some $i\in I.$
\eqref{5.} further indicates that
\[
\sup_{\xi\in \mathbb{R}}\psi_{i}(\xi) > u_i^*,
\]
and the existence of nonmonotone traveling wave solutions follows from \eqref{5.1}.

To further illustrate our conclusions, we also give some numerical simulations. Take
\begin{equation*}
\begin{cases}
0.0001\phi _1^{\prime \prime }(\xi)-\phi _1^{\prime }(\xi)+0.1\phi
_1(\xi)\left[
1-\phi _1(\xi)-0.55\phi _2(\xi)\right] =0, \\
0.05\phi _2^{\prime \prime }(\xi)-\phi _2^{\prime }(\xi)+0.5\phi
_2(\xi)\left[ 1-\phi _2(\xi)-0.75\phi _1(\xi)\right] =0
\end{cases}
\end{equation*}
such that
\[
k_1\approx 0.7659,k_2\approx 0.4255.
\]
Then the nonmonotonic traveling wave solutions are presented in Figures \ref{fig2}-\ref{fig3}.

\begin{figure}
\begin{center}
\includegraphics[width=110mm]{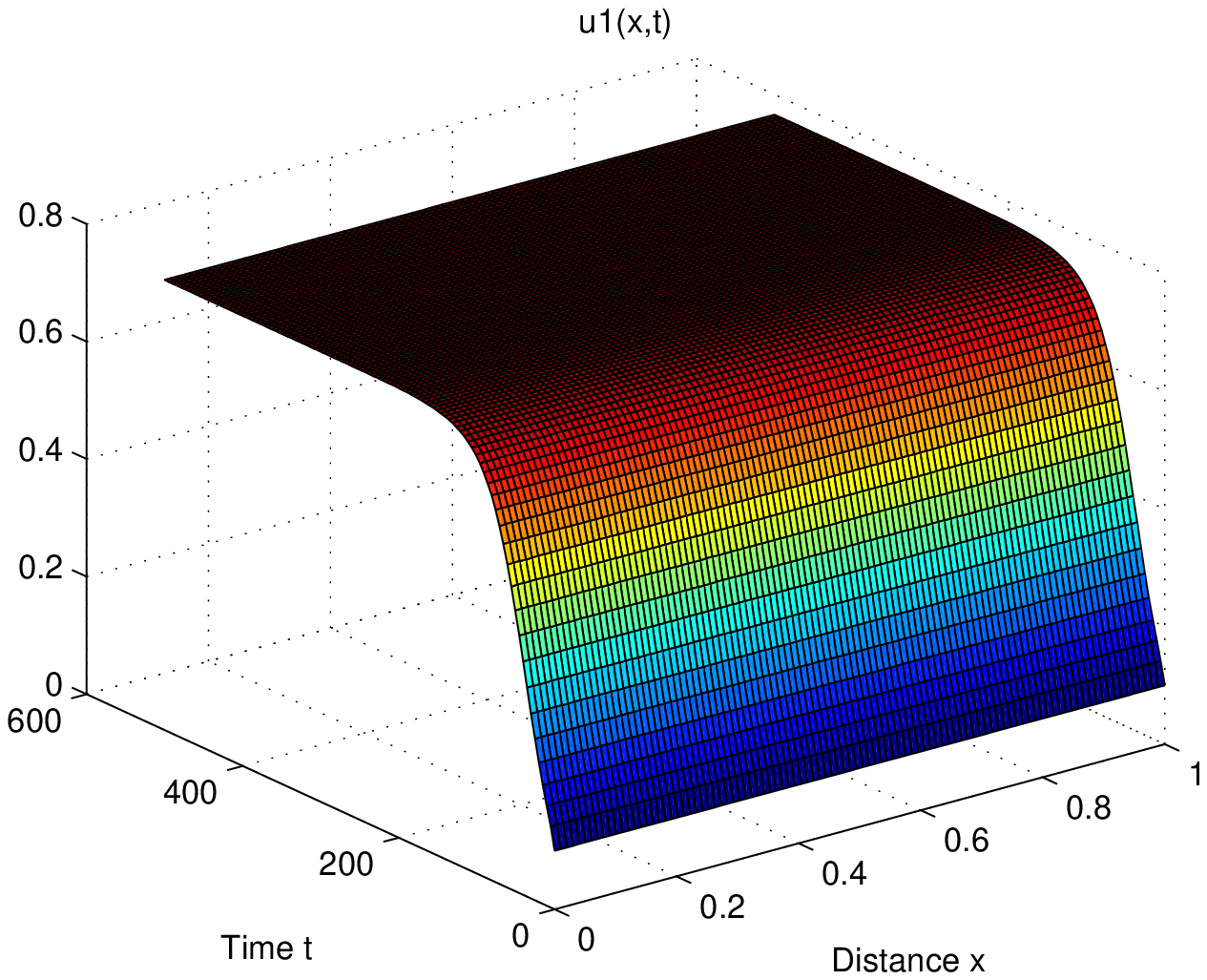}
\caption{The traveling wave solution $u_1 (x, t)$.}\label{fig2}
\end{center}
\end{figure}
\begin{figure}
\begin{center}
\includegraphics[width=110mm]{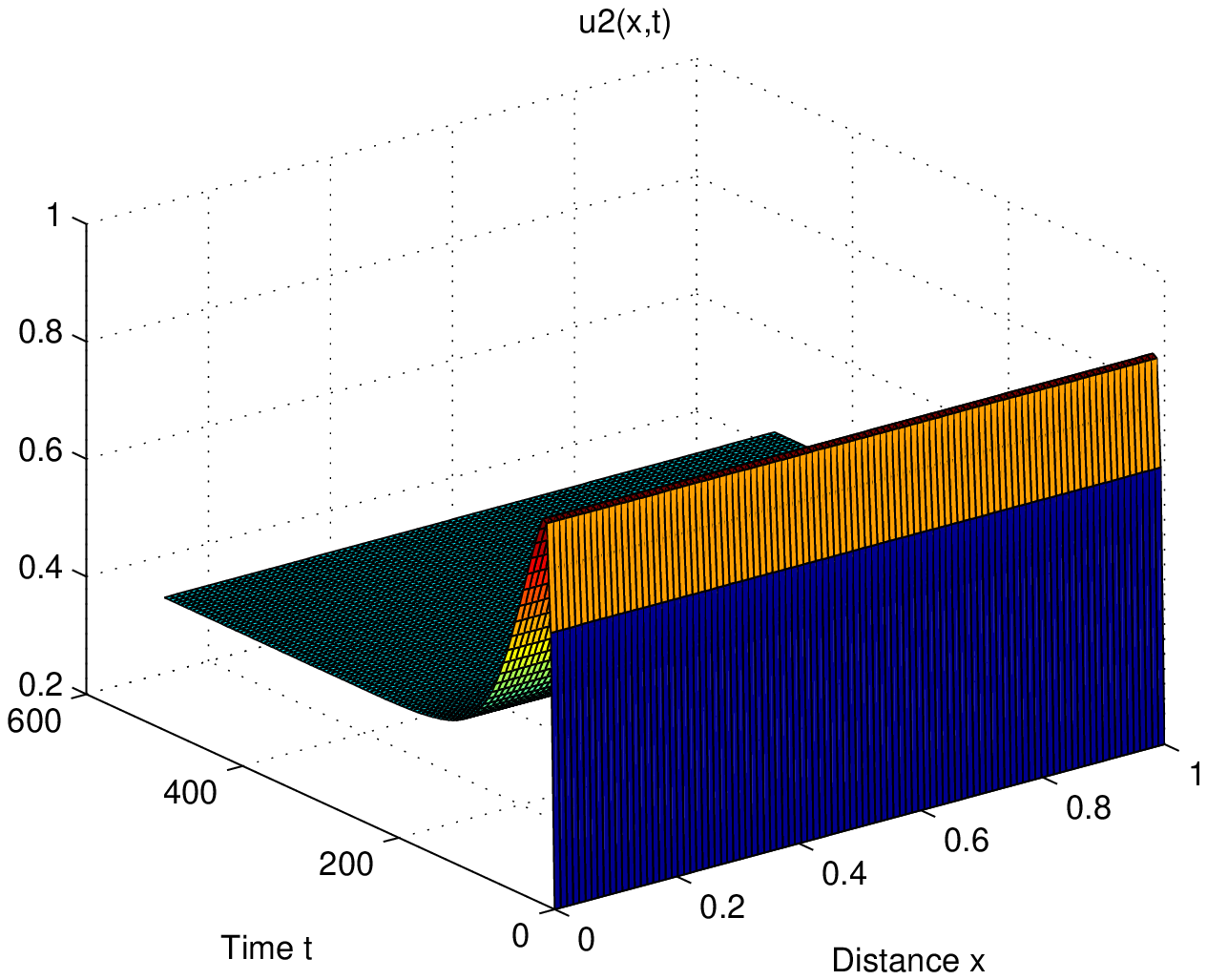}
\caption{The traveling wave solution $u_2 (x, t)$.}\label{fig3}
\end{center}
\end{figure}

\subsection{Further Results for $n=1$}
\noindent

It is difficult to consider the existence of \eqref{5.0}-\eqref{5.1} if $c=\max_{i\in I}\{2 \sqrt{d_ir_i}\}.$ But for  $n=1,$ we can obtain the existence of traveling wave solutions by passing to a limit function. When $n=1,$  \eqref{1} becomes
\begin{equation}
\dfrac{\partial v(x,t)}{\partial t}=d\Delta v(x,t)+rv(x,t)\left[
1-\int_{-\tau }^{0}v(x,t+s)d\overline{\zeta} (s)\right] ,  \label{0}
\end{equation}
herein $v\in \mathbb{R}, d>0,r>0, x\in\mathbb{R},t>0$ and
\[
\overline{\zeta}(s) \text{  is nondecreasing on }[-\tau,0] \text{
and }\overline{\zeta}(0)-\overline{\zeta}(-\tau)=1.
\]
In particular, we shall suppose that $\tau$ is the real
maximum delay involved in \eqref{0}, namely, there is no
$\tau_0<\tau$ such that $\int_{-\tau_0}^0 d\overline{\zeta} (s)=1.$
Denote
\[
b=\overline{\zeta}(0)-\overline{\zeta}(0-)
\]
and set
\[
\zeta(s)= \begin{cases}\overline{\zeta}(s),\,\,\,s\in \lbrack
-\tau ,0),\\
\overline{\zeta}(0-),\,\,\,s=0. \end{cases}
\]
Then \eqref{2} implies that \begin{equation}\label{3} b\in
(1/2,1]\end{equation} and \eqref{3}
will be imposed in this subsection.

Let $v(x,t)=\phi(x+ct)$ be a traveling wave solution of
\eqref{0}. Then it satisfies
\begin{equation}
d\phi ^{\prime \prime }(\xi )-c\phi ^{\prime }(\xi )+r\phi (\xi
)\left[ 1-b\phi (\xi )-\int_{-\tau }^{0}\phi(\xi+cs)d\zeta
(s)\right]=0, \,\, \xi\in\mathbb{R}  \label{2.0}
\end{equation}
and the asymptotic boundary conditions
\begin{equation}
\lim_{\xi \rightarrow -\infty }\phi (\xi )=0,\,\,\,\lim_{\xi \rightarrow
\infty }\phi (\xi )=1.  \label{2.1}
\end{equation}
Clearly, if $c>2\sqrt{dr}$ or $c< 2\sqrt{dr},$ then the existence or nonexistence of \eqref{2.0}-\eqref{2.1} has been addressed by the previous results. The following result deals with the case when $c=2\sqrt{dr}.$
\begin{theorem}\label{th3}
When $c=2 \sqrt{dr},$ \eqref{0} also has a  positive
traveling wave solution connecting $0$ with $1.$
\end{theorem}
\begin{proof}
Let $\{c_n\}$ be a decreasing sequence with $c_n<4 \sqrt{dr}$ and
$c_n\to 2\sqrt{dr}, n\to \infty$. Then for each $c_n,$ \eqref{0} has
a positive traveling wave solution connecting $0$ with $1,$ denoted
by $\phi^n(\xi).$ It follows that  $0\le \phi^n(\xi)\le
1/b,\xi\in\mathbb{R},n\in\mathbb{N}.$ Note that a traveling wave
solution is invariant in the sense of phase shift, so we assume
that
\begin{equation}\label{0000}
\phi^n(0)=\frac{2b-1}{8b},\,\,\phi^n(\xi)<\frac{2b-1}{8b}, \,\, \xi<0.
\end{equation}
By \eqref{2.1}, we know that \eqref{0000} is admissible. For $n\in
\mathbb{N}, \xi\in \mathbb{R},$ it is evident that $\phi^n(\xi)$ are
equicontinuous and bounded in the sense of the supremum norm. By Ascoli-Arzela lemma and a nested subsequence argument,  $\{\phi^n(\xi)\}$ has a subsequence, still
denoted by $\{\phi^n(\xi)\},$ such that
\[
\phi^n(\xi) \to \phi(\xi),\,\, n\to\infty
\]
for a continuous function $\phi(\xi)$.
We  see that the above convergence is
pointwise on $\mathbb{R}$  and is also uniform on any bounded interval of
$\mathbb{R}.$

Note that
\[ \min \{e^{\gamma _{1}(c_{n})(\xi -s)},e^{\gamma
_{2}(c_{n})(\xi
-s)}\}\rightarrow \min \{e^{\gamma _{1}(2\sqrt{dr})(\xi -s)},e^{\gamma _{2}(2%
\sqrt{dr})(\xi -s)}\},n\rightarrow \infty ,
\]%
for any given $\xi \in \mathbb{R},$ and the convergence in $s$ is
uniform for $s\in \mathbb{R}.$ Applying the dominated convergence
theorem in $F,$ we know that $\phi(\xi)$ is a fixed point of $F$ and
a solution to \eqref{2.0}.

By \eqref{2.1} and \eqref{0000}, $\phi$ also satisfies
\[
\phi(0)=\frac{2b-1}{8b},\,\,\phi(\xi)\le \frac{2b-1}{8b}, \,\, \xi<0,\,\,
\phi\in X_{[0,1/b]}.
\]
Namely, $\phi(\xi)$ is a  positive solution to \eqref{2.0}, which is uniformly continuous for $\xi\in\mathbb{R}$. Now, we are in a position to verify the
asymptotic boundary conditions \eqref{2.1}. Due to $\phi (0)>0$ and Lemma \ref{le1.1}, $\liminf_{\xi\to\infty}\phi(\xi) >0$ holds.  By Theorem \ref{th8},
we obtain $\lim_{\xi\to\infty}\phi(\xi)=1.$

Define
\[
\limsup_{\xi\to -\infty}\phi(\xi)=\phi^+, \,\,\liminf_{\xi\to
-\infty}\phi(\xi)=\phi^-.
\]
It is clear that
\[
0\le \phi^-\le \phi^+\le  \frac{2b-1}{8b}.
\]
If $\phi^->0,$ then the dominated convergence
theorem in $F$ implies that
\[
4\phi^-\ge 4 \phi^-+ \phi^-(1-b\phi^--(1-b)\phi^+),
\]
which is impossible by \eqref{0000}. Hence, $\phi^-=0$ follows.

If $\phi^+>0,$ then there exist $\{\xi _j\}, j\in \mathbb{N}$ with $\xi_j\to -\infty, j\to \infty$ such that
$\phi(\xi_j)\to \phi^+, j\to\infty .$ Using the uniform continuity
of $\phi,$ there exists $\delta >0$ such that
\begin{equation}\label{2.9}
\frac{2b-1}{8b} \ge  \phi(\xi)\ge \phi^+/2,\,\, \xi\in (\xi_j-\delta,
\xi_j+ \delta), \,\,j\to\infty.
\end{equation}

We now return to the Fisher equation
\begin{equation*}
\begin{cases}
\frac{\partial z(x,t)}{\partial t}=d \Delta z(x,t)+rz (x,t)\left[
1-(1-b)/b-bz(x,t)\right] ,\\
z(x,0)=z(x),
\end{cases}
\end{equation*}
of which an upper solution is $\phi(x+ct)$ if $z(x)\le \phi(x).$ Let
$z(x)$ satisfy
\begin{description}
\item[(z1)] $z(x)\in X,$
\item[(z2)] $z(x)=\phi^+/2 ,$ $x\in [-\delta/2,  \delta/2],$
\item[(z3)] $z(\pm \delta)=0 $ and $z(x)$ is decreasing (increasing)
if $x\in (\delta/2, \delta)$ ($x\in (-\delta, -\delta/2)$),
\item[(z4)] $z(x)=0$ if $|x|>\delta.$
\end{description}
Because of $\phi(0)=  \frac{2b-1}{8b}\ge \phi^+$ and the uniform continuity of $\phi (\xi),$ we see that $\delta
>0$ is admissible.
Then there exists $T>0$ such that
\begin{equation}\label{2.10}
z(0,t)>\frac{2b-1}{4b}>\phi^+,\,\,t\ge T \text{ (see Lemma
\ref{le1.1}).}
\end{equation}

For each $j\in \mathbb{N},$ there exists $j' \in \mathbb{N}$ such that
$\xi_{j'}<\xi_j-2\sqrt{dr}T- 2\delta .$ When $\xi_{j}$ and $\xi_{j'}$
satisfy \eqref{2.9}, then \eqref{2.10} and Lemma \ref{le1.1}
imply that
\[\phi(\xi_j)>\frac{2b-1}{4b}>\phi^+,\]
which is impossible by \eqref{2.9} and the arbitrariness of $j$.
Therefore, $\phi^+=0$. The proof is complete.
\end{proof}

\section{Discussion}
\noindent

Delay is a very common process in many biological and physical phenomena and
many important and realistic models with delay have been proposed to describe various problems in applied subjects.
The fundamental theory of delayed differential equations have been well-developed,
we refer to the monographs of Hale and Verduyn Lunel \cite{hale} and Wu \cite{wu}.
It is well-known that the dynamics between the delayed
and undelayed systems may be significantly different; for instance, the
delayed Logistic equation or Hutchinson equation exhibits
nontrivial periodic solutions while all the nonnegative solutions of
the Logistic equation converge to the positive steady state.
Moreover, to study delayed systems, more complex phase spaces
than that of the corresponding undelayed systems are required. The investigation of
traveling wave solutions of delayed systems is also more difficult
than that of the corresponding undelayed systems, at least the phase plane method
which is powerful in studying undelayed systems meets some
difficulties in the study of delayed systems.

Of course, if a system is (local) quasimonotone, then the classical
theory established for monotone semiflows is applicable and there are plentiful results.
For example, the existence, nonexistence,
minimal wave speed,  uniqueness and stability of traveling wave
solutions have been widely studied and many sharp results have been
established, see Liang and Zhao \cite{liangzhao}, Schaaf \cite{schaaf}, Thieme and Zhao \cite{t},
Smith and Zhao \cite{smithzhao}, and Wang et al. \cite{wangli1,wang3}.
If the system is not quasimonotone, then the study becomes harder
and some new phenomena can occur, e.g., the existence of nonmonotone
traveling wave solutions in scalar equations, see Faria and Trofimchuk \cite{fa1}. When the delay is small
enough, some nice results on traveling wave solutions can also be
obtained by different techniques such as exponential ordering,
perturbation and so on, see Ai \cite{ai}, Fang and Wu \cite{fa}, Lin et al. \cite{llm},
Ou and Wu \cite{ou}, Wang et al. \cite{wang1}, and Wu and Zou \cite{wuzou2}.

However, if the delay is large, these techniques cannot deal with the
traveling wave solutions of delayed systems including \eqref{1} and \eqref{0}. In this paper, we
applied generalized upper and lower solutions to seek after the positive traveling wave
solutions.  Since these systems do not satisfy the
quasimonotone condition, the limit behavior of traveling wave
solutions cannot be obtained by the monotonicity of traveling wave
solutions and the dominated convergence theorem.  In particular, for
the case $c=2\sqrt{dr}$ in \eqref{0}, the asymptotic behavior  of
traveling wave solutions cannot be considered by the techniques of
monotone traveling wave solutions in Liang and Zhao \cite{liangzhao} and Thieme and Zhao \cite{t}. In this paper, we studied the asymptotic behavior of traveling wave solutions of  \eqref{1} and \eqref{0} by combining the idea of contracting rectangles with the theory of asymptotic spreading.

Our results imply that if there exists instantaneous self-limitation effect, then the \emph{large delays} appearing in the
intra-specific competition terms may not affect the persistence of traveling wave solutions.
However, very likely large delay may also lead to some significant
differences between the traveling wave solutions of delayed and
undelayed systems, e.g., the nonexistence of monotone traveling wave solutions of \eqref{0}, which will be
investigated in our future studies.

\vspace*{0.3cm}
\noindent{\bf Acknowledgements.} The authors would like to thank an anonymous reviewer for his/her helpful
comments and Yanli Huang for her valuable suggestions. This research was partially supported by the the National Natural Science
Foundation of China (11101194) and the National Science Foundation (DMS-1022728).

\end{document}